\documentclass[12pt]{amsart}
\usepackage{amsmath,amssymb,latexsym,cancel,rotating}

\numberwithin{equation}{section}

\newtheorem{theorem}{Theorem}[section]
\newtheorem{proposition}[theorem]{Proposition}
\newtheorem{conjecture}[theorem]{Conjecture}
\newtheorem{corollary}[theorem]{Corollary}
\newtheorem{lemma}[theorem]{Lemma}

\theoremstyle{definition}

\newtheorem{remark}[theorem]{Remark}

\newtheorem{definition}[theorem]{Definition}

\input xy
\xyoption{poly}
\xyoption{2cell}
\xyoption{all}

\def\ZZ{\mathbb{Z}}

\def\QQ{\mathbb{Q}}

\def\Acal{\mathcal{A}}

\def\Fcal{\mathcal{F}}

\def\Pcal{\mathcal{P}}
\def\Qcal{\mathcal{Q}}

\def\Tcal{\mathcal{T}}
\def\Xcal{\mathcal{X}}

\def\Fbb{\mathbb{F}}
\def\Sbb{\mathbb{S}}

\def\dim{\text{dim}}

\def\gcd{\text{gcd}}
\def\Hom{\text{Hom}}
\def\Ext{\text{Ext}}

\def\half{\frac{1}{2}}

\renewcommand{\eqref}[1]{{\rm (\ref{#1})}}

\begin{document}

\title[On Quantum Analogue of The Caldero-Chapoton Formula]
{On Quantum Analogue of The Caldero-Chapoton Formula}

\author{Dylan Rupel}
\address{\noindent Department of Mathematics, University of Oregon,
 Eugene, OR 97403}
\email{drupel@uoregon.edu}

\date{April 15, 2010}

\maketitle

\section{Introduction and Main Results}

Quantum cluster algebras  were introduced by Berenstein and Zelevinsky in
\cite{berzel} as a noncommutative analogue of cluster algebras (see \cite{ca2})
to lay the groundwork for a study of the canonical basis.  A quantum cluster
algebra $\Acal_q(Q)$ of rank $n$ is generated by a (possibly
infinite) set of generators called the \textit{cluster variables} inside an
ambient skew-field $\Fcal_q$.  The goal of this paper is to explicitly
compute all cluster variables for a large class of quantum cluster algebras.

We start with rank $2$ quantum cluster algebras.  Let $q$ be a formal variable.  Let $\Tcal_{q}$ be the
$2$-dimensional quantum torus, i.e.,
$\Tcal_{q}=\ZZ[q^{\pm1/2}]\langle X_1^{\pm1}, X_2^{\pm1}: X_1X_2=qX_2X_1\rangle$ and
let ${\mathcal F}_{q}$ be the skew field of fractions of $\Tcal_{q}$.  For
$b,c\in\ZZ_{>0}$ define the quantum cluster algebra $\Acal_{q}(b,c)$ to be the
$\ZZ[q^{\pm1/2}]$-subalgebra of ${\mathcal F}_{q}$ generated by the cluster variables
$X_k$, $k\in\ZZ$, defined recursively by 
\begin{equation}\label{eq: clRel}
X_{m-1}X_{m+1}=
\begin{cases}q^{b/2}X_m^b+1 & \text{ if $m$ is odd}\\ q^{c/2}X_m^c+1 & \text{ if
$m$ is even.}\\\end{cases}
\end{equation}
In what follows we will routinely specialize $q$ to a positive real number, usually the size of a finite field.  The quantum Laurent phenomenon (\cite[Corollary 5.2]{berzel}) implies that each
$X_k$, in fact, belongs to the subring $\Tcal_{q}$ of $\Fcal_{q}$ and thus $\Acal_q(b,c)$ is contained in $\Tcal_{q}$. However, the explicit computation of each $X_k$ as a Laurent polynomial in $X_1$ and $X_2$ is a non-trivial task.

We will use the following notation throughout the paper.  Define $X^{(a_1,a_2)}\in\Tcal_{q}$ by the formula $X^{(a_1,a_2)}:=q^{-\half a_1a_2}X_1^{a_1}X_2^{a_2}$.  Also define the symmetrized quantum binomial coefficient 
\begin{equation}\label{eq:sym-bin}
{n \brack r}_q:=\frac{(q^n-q^{-n})(q^{n-1}-q^{-n+1})\cdots(q^{n-r+1}-q^{-n+r-1})}{(q^r-q^{
-r})\cdots(q-q^{-1})}.
\end{equation}
The following result shows that all cluster variables $X_k$ for $b=c=2$ are computable combinatorially.  
\begin{proposition}
\label{th:XZ-formula}
For every $n \geq 0$, we have in $A_q(2,2)$:
\begin{eqnarray}
&\label{eq:X-n-formula}X_{-n} &\textstyle  =  X^{(n+2,-n-1)} + \sum_{p + r \leq
n} {n-r \brack p}_q{n+1-p \brack r}_q X^{(2r-n,2p-n-1)};\\
&\label{eq:Xn-formula}X_{n+3} &\textstyle  =  X^{(-n-1,n+2)} + \sum_{p + r \leq
n} {n-r \brack p}_q{n+1-p \brack r}_q X^{(2p-n-1,2r-n)}.
\end{eqnarray}
\end{proposition}

This result was proved independently in a recent preprint \cite{lampe}, we
present our proof for the convenience of the reader. Setting $q=1$ we recover
the Caldero-Chapoton-Zelevinsky formulas for $x_k$ from \cite{calzel}.
This suggests that Proposition~\ref{th:XZ-formula} should have a categorical and
quiver-theoretic interpretation in the form of a deformation of the
Caldero-Chapoton formula \cite[Eq. 14]{caldchap}.

In order to state our generalization of Proposition \ref{th:XZ-formula} to any  algebra
$\Acal_q(b,c)$, we need some quiver-theoretic notation.  Let $d=\gcd(b,c)$ and let $Q_{b,c}$ be the quiver
\xymatrix{\circ_1 \ar[r]^d & \circ_2} with $d$ arrows.  Let $\Fbb$ be a finite field and $\bar{\Fbb}$ an algebraic closure of $\Fbb$.  For any integer $n>0$ we will denote by $\Fbb_n$ the degree $n$ field extension of $\Fbb$ inside $\bar{\Fbb}$.  We define a valued representation $V$ of $Q_{b,c}$ by assigning an $\Fbb_c$-vector space $V_1$
to the first vertex, an $\Fbb_b$-vector space $V_2$ to the second vertex,
and an $\Fbb_d$-linear map $\varphi_i:V_1\to V_2$, $i=1,2,\ldots,d$, to each
arrow.

Let $V$ be a valued representation of $Q_{b,c}$ with dimension vector $[V] =
(v_1,v_2)$.  For ${\bf e}=(e_1,e_2)\in\ZZ_{\ge0}^2$, denote by $Gr_{\bf e}(V)$
the set of all subrepresentations $M$ of $V$ (i.e., $M=(M_1,M_2)$, where
$M_\delta$ is a subspace of $V_\delta$, $\delta=1,2$ such that
$\varphi_i(M_1)\subset M_2$ for $i=1,2,\ldots,d$) with $[M]={\bf e}$.  This is a finite set since $V$ is finite.  For each
valued representation $V$ of $Q_{b,c}
$ we define the element $X_V$ of the quantum torus $\Tcal_{|\Fbb|}$ by 
\begin{equation}
\label{eq:XV rank 2}
X_V=\sum_{{\bf e}} |\Fbb|^{-\half d_{{\bf e}}^V}|Gr_{{\bf e}}(V)|
X^{(-v_1+bv_2-be_2,ce_1-v_2)}
\end{equation}
where $d_{{\bf e}}^V=ce_1(v_1-e_1)-b(ce_1-e_2)(v_2-e_2)$.  When $b=c$, the valued representations of $Q_{b,b}$ are just the ordinary $\Fbb_{b}$-representations of $Q_{b,b}$ and, as we will demonstrate below, this formula gives a
deformation of the Caldero-Chapoton formula.

Let $C=\left(\begin{array}{cc} 2 & -b\\ -c & 2\end{array}\right)$ be a Cartan matrix and let $\Phi$
be the associated root system with simple roots $\{\alpha_1, \alpha_2\}$. We will label all negative real roots of $\Phi$ by $\ZZ\setminus \{1,2\}$ recursively as follows:
$$\alpha_{m-1}+\alpha_{m+1}=\begin{cases} 
ba_m & \text{if $m$ is odd} \\
ca_m & \text{if $m$ is even} \\
\end{cases} $$
for $m\in \ZZ\setminus \{1,2\}$ 
with the convention $\alpha_0=-\alpha_2$, $\alpha_3=-\alpha_1$. 
 
Then denote by $V_{(m)}$ the unique indecomposable valued
representation of $Q_{b,c}$ with dimension vector $-\alpha_m$ (see e.g., \cite[Theorem 16]{hub}). 

\begin{theorem}\label{rank2}
For any  $b,c\in\ZZ_{>0}^2$ and  each $m\in\ZZ\setminus \{1,2\}$, the $m$-th cluster variable $X_m$ of $\Acal_{|\Fbb|}(b,c)$ equals $X_{V_{(m)}}$.
\end{theorem}

We will prove Theorem \ref{rank2} in Section \ref{def}.
In section \ref{finitetype} we will illustrate Theorem 1.2 by computing all $X_k$ of finite types, i.e., when $bc\le 3$.

Now we consider a general class of rank $n\ge 2$ quantum cluster
algebras attached to valued quivers on $n$ vertices.  

Let $Q$ be a quiver without vertex loops or 2-cycles.  Suppose $Q$ has vertices $[1,n]$ and $d_{ij}$ edges from $i$ to $j$, and let $d_i$, $i=1,\ldots,n$ be \emph{valuations} on the vertices. We will call such a quiver a \emph{valued quiver}. 
Define the matrix $B=B_Q=(b_{ij})$ by 
\[b_{ij}=\begin{cases} d_{ij}d_j/gcd(d_i,d_j) & \text{ if $i\to j$ in
$Q$}\\ -d_{ij}d_j/gcd(d_i,d_j) & \text{ if $j\to i$ in $Q$}\\ 0 &
\text{ otherwise.}\end{cases}\]
For $D=diag(d_1,\ldots,d_n)$ we have $DB$ is skew-symmetric, i.e.
$d_ib_{ij}=-d_jb_{ji}$ for all $i$ and $j$.  From the pair $(B,D)$ we can construct a valued quiver $Q_B$ having vertices $[1,n]$ with valuations $d_i$ and, when $b_{ij}>0$, $Q_B$ has $gcd(|b_{ij}|,|b_{ji}|)$ edges from vertex $i$ to vertex $j$.  This gives a one-to-one correspondence between valued quivers and skew-symmetrizable matrices with symmetrization data.  Thus we will freely identify the valued quiver $Q$ with the pair $(B,D)$ and call the pair $(B,D)$ a \emph{valued quiver}.
We will call the valued quiver $Q$ \emph{invertible} if $B$ is invertible.  If $B$ is not invertible, we replace the valued quiver $(B,D)$ by the invertible valued quiver $\tilde{Q}=(\tilde{B},D\oplus D)$ with $2n$ vertices as in section \ref{def}, we will call $Q$ the principal subquiver of $\tilde{Q}$. Thus we will assume that $n$ is even and valued quiver will always mean invertible valued quiver.

We define valued representations of $Q$ by assigning an $\Fbb_{d_i}$-vector space to
each vertex $i$ and an $\Fbb_{gcd(d_i,d_j)}$-linear map to each edge $i\to
j$.  Denote the category of all finite-dimensional valued representations of $Q$ by $rep~Q$.  It is well-known (see e.g., \cite{schiffmann}) that $rep~Q$ is a \emph{length category}, i.e. the Grothendieck group
$\Qcal$ of $rep~Q$ is a free abelian group
generated by the classes $\alpha_i=[S_i]$, $i=1,\ldots,n$ of simple representations associated to the vertices.  In particular, the class $[V]\in \Qcal$ of an object $V$ is
naturally identified with the dimension vector of $V$. The Euler form on
$\Qcal$ is given by bilinearly extending the following formula 
\[\langle\alpha_i,\alpha_j\rangle=
\begin{cases}
d_i & \text{if $i=j$}\\ 
-\max(0,d_ib_{ij}) & \text{if $i\ne j$.}
\end{cases}\]
Furthermore, abbreviate $\alpha_i^\vee:=\frac{1}{d_i}\alpha_i$ and note that 
$\langle \alpha_i^\vee,{\bf e}\rangle$ and $\langle{\bf e}, 
\alpha_i^\vee\rangle$ are integers for all ${\bf e}\in
\Qcal$. Then for ${\bf e} \in \Qcal$ 
define vectors ${}^*{\bf e},{\bf e}^*\in \ZZ^n$ by 
\[{}^*{\bf e}=(\langle \alpha_1^\vee,{\bf e}\rangle,\ldots,\langle
\alpha_n^\vee,{\bf e}\rangle),~{\bf e}^*=(\langle {\bf
e},\alpha_1^\vee\rangle,\ldots,\langle {\bf e},\alpha_n^\vee\rangle).\]

Denote $\Lambda=\Lambda_Q := -DB^{-1}$.  By definition $\Lambda$ is skew-symmetric and satisfies $\Lambda B = -B^T\Lambda = -D$. 
Write $\Lambda=(\lambda_{ij})$ and let $d$ be the least common multiple of all the denominators of the $\lambda_{ij}$'s. Let $\Tcal_{\Lambda_Q,q}$ denote the \mbox{$n$-dimensional} quantum torus, i.e. \[\Tcal_{\Lambda_Q,q}=\ZZ[q^{\pm\frac{1}{2d}}]\langle X_1^{\pm
1},\ldots,X_{n}^{\pm 1}| X_iX_j=q^{\lambda_{ij}} X_jX_i\rangle.\]  
For each ${\bf a}=(a_1,\ldots,a_n)\in \ZZ^n$ we define a monomial $X^{({\bf a})}\in
\Tcal_{\Lambda_Q,q}$ by: 
$$X^{({\bf a})}:=q^{-\frac{1}{2}\sum\limits_{i<j} \lambda_{ij}a_i a_j} X_1^{a_1}\cdots
X_n^{a_n}.$$   

For $V\in rep~Q$ and ${\bf e}\in
\Qcal$ define $Gr_{\bf e}(V)$ to be the set of all
subobjects $W$ of $V$ such that $[W]=\bf e$.  Sometimes we will think of this as the set of all short exact sequences $\{0\to W\subset V\to V/W\to 0 : [W]={\bf e}\}$.  Note
that $Gr_{\bf e}(V)$ is finite since $V$ is a finite set. Define the element $X_V\in \Tcal_{\Lambda_Q,|\Fbb|}$
by the formula:
\begin{equation}\label{qcc-formula2}
X_V=\sum\limits_{{\bf e}\in \Qcal} |\Fbb|^{-\half\langle{\bf
e},[V]-{\bf e}\rangle}|Gr_{{\bf e}}(V)| X^{(B{\bf e}-{}^*[V])}.
\end{equation}
Note that $B{\bf e}={}^*{\bf e}-{\bf e}^*$.
This is equivalent to the following formula for $X_V$ which the reader may find useful:
$$X_V=\sum\limits_{M\subset V} |\Fbb|^{-\half\langle [M],[V/M]\rangle} X^{(-[M]^*-{}^*[V/M])}.$$
It is easy to see that when $n=2$ equation \eqref{qcc-formula2} specializes to equation \eqref{eq:XV rank 2}.
When $d_1=\cdots=d_n=\delta$, i.e. the quiver is equally valued it is known (see \cite[Corollary 4]{caldrein}, footnote 5 on page 6 of \cite{naka}, and the corrected proof in \cite{qin})
that for $V$ exceptional and indecomposable $Gr_{{\bf e}}(V)$ is the set of $\Fbb_{\delta}$ points of an algebraic variety of dimension $\langle{\bf
e},[V]-{\bf e}\rangle/\delta$ and $|Gr_{{\bf e}}(V)|$ is given by a positive polynomial
in $|\Fbb|^\delta$. 

For a sink or source $i$ of $Q$ denote by $\mu_iQ$ the valued quiver $(\mu_iB,D)$ where we change the sign of each entry of $B$ in row $i$ or column $i$.  This is equivalent to reversing all arrows with vertex $i$ as source or target.  We will call a sequence of vertices $k_1$, $k_2$, \ldots, $k_{r+1}$ in $Q$ \emph{admissible} if the following hold:
\begin{itemize}
\item $k_i\ne k_{i+1}$ for each $i$;
\item $k_1$ is a sink or source in $Q$;
\item for each $1\le i\le r-1$, vertex $k_{i+1}$ is a sink or source in the quiver $\mu_{k_i}\mu_{k_{i-1}}\cdots\mu_{k_1}Q$.
\end{itemize}
Note that $k_{r+1}$ does not have to be a sink or a source.

Let $C$ denote the $n\times n$ Cartan counterpart to $B$, i.e. define 
\[c_{ij}=\begin{cases}2 & \text{ if } i=j\\ -|b_{ij}| & \text{ if } i\neq
j.\end{cases}\]
Let $\Phi$ denote the root system associated to $C$, see \cite{kac}.  We will
identify $\Qcal$ with the root lattice of $\Phi$ by taking
$\Pi=\{\alpha_1,\ldots,\alpha_n\}$ to be the set of simple roots in $\Phi$. 
Define simple reflections $\sigma_i$ in the Weyl group of $\Phi$ by setting
$\sigma_i(\alpha_j)=\alpha_j-c_{ij}\alpha_i$ and extending linearly (this is the
correspondence we will use between the root system $\Phi$ and $C$).

Denote by $\Acal_q(Q)\subset \Tcal_{\Lambda_Q,q}$ the quantum cluster
algebra corresponding to the invertible valued quiver $Q$ (see
Section~\ref{def} for details).  

\begin{remark}\label{rmk:coeff}
We restrict our attention to invertible valued quivers for simplicity.  However the results below hold for any compatible pair $(\Lambda,\tilde{B})$ where $Q$ is the valued quiver associated to the principal part of $\tilde{B}$.
\end{remark}

Our main result is the following

\begin{theorem}\label{main}
Let $V_\alpha$ be the unique indecomposable object of $rep~Q$ with dimension vector $\alpha$ given by $\alpha=\sigma_{k_1}\sigma_{k_2}\cdots\sigma_{k_r}(\alpha_{k_{r+1}})$ 
where $k_1,k_2,\ldots,k_{r+1}$ is an admissible sequence in $Q$.  Then $X_{V_\alpha}$ is a cluster variable of $\Acal_{|\Fbb|}(Q)$. Conversely if the quiver $Q$ 
is acyclic, each cluster variable of $\Acal_{|\Fbb|}(Q)$ belonging to any acyclic cluster is of the form $X_{V_\alpha}$ for some $\alpha$ as above.
\end{theorem}

We will prove Theorem \ref{main} in Section \ref{def}.  

\begin{theorem}\label{cor:fintype}
If the valued quiver $Q$ is a Dynkin diagram, then each cluster variable of $\Acal_{|\Fbb|}(Q)$ is of the form $X_N$ for some indecomposable $N\in rep~Q$.
\end{theorem}

We will prove Theorem \ref{cor:fintype} in Section \ref{def}.  For any invertible valued quiver $Q$ let $\Tcal_{DB,q}=\ZZ[q^{\pm\half}]\langle Z_1^{\pm1},\ldots,Z_{n}^{\pm 1}| Z_iZ_j=q^{d_ib_{ij}} Z_jZ_i\rangle$ be the quantum torus associated to the skew-symmetric matrix $DB$.  It is easy to see that the assignment $Z_i\mapsto X^{({\bf b}^i)}$ defines an embedding of quantum tori which extends to an embedding of skew-fields $j: \Fcal_{DB,q}\to \Fcal_{\Lambda_Q,q}$.  We say that an element $X\in\Fcal_{\Lambda_Q,q}$ is \emph{$B$-compatible} if $X$ can be written as $j(F_X)X^{({\bf g}_X)}$ for some $F_X\in\Fcal_{DB,q}$ and ${\bf g}_X\in\ZZ^n$.  We refer to $F_X$ as an $F$-factor for $X$ thus generalizing the definition of $F$-polynomial from \cite[Theorem 5.3]{tran}.  It follows from \cite[Theorem 5.3]{tran} that each cluster monomial $X$ is $B$-compatible and (under appropriate choice of ${\bf g}_X$) the $F$-factor $F_X$ is actually a polynomial.  The following result justifies these definitions.

\begin{proposition}\label{prop:compatible}\mbox{}
\begin{enumerate}
\item\label{comp} The mutation of $B$-compatible elements in direction $k$ gives $\mu_kB$-compatible elements.
\item\label{f-factor} For $V\in rep~Q$, $X_V$ is $B$-compatible and we have $$F_{X_V}=1+\sum\limits_{\bf e\in\Qcal\setminus \{0\}}q^{\half\langle {\bf e},{\bf e}\rangle}|Gr_{\bf e}(V)|Z^{({\bf e})}.$$
\end{enumerate}
\end{proposition}

\begin{remark}
Following up on Remark~\ref{rmk:coeff}, since all our results hold for principal coefficients, Proposition~\ref{prop:compatible}\eqref{f-factor} together with \cite[Theorem 5.3(2)]{tran} directly let one work with any coefficients.
\end{remark}

\begin{corollary}\label{cor:fin-fpoly}
Let $Q$ be of finite type.  Then for any indecomposable $V\in rep~Q$ we have $$F_{X_V}\in 1+\sum\limits_{{\bf e}\in\ZZ_{\ge0}^n\setminus\{0\}} q^{\half}\ZZ[q^{\half}]Z^{({\bf e})}.$$
\end{corollary}

\begin{remark}
For simply laced types the positivity of coefficients of $F_{X_V}$ follows from the positivity of grassmannians from \cite{qin}.  See Section~\ref{g2} for positivity in type $G_2$.  Positivity for selected clusters in types $B_n$ and $C_n$ follow from \cite[Theorem 6.1]{tran2}.
\end{remark}

Based on the above results, we now conjecture a general deformation of the Caldero-Chapoton formula for any quantum cluster algebra with acyclic seed. 

\begin{conjecture}\label{qccconj}
Let $Q$ be an acyclic valued quiver.  Suppose 
$V\in rep~Q$ is an indecomposable exceptional valued representation.
Then $X_V$ is a cluster variable in $\Acal_{|\Fbb|}(Q)$.
\end{conjecture}

In Corollary~\ref{acyclic} we show that all cluster variables in almost acyclic clusters (clusters which are one mutation away from an acyclic cluster) are of the form $X_V$ for some $V\in rep~Q$.
We present further evidence for this conjecture in section \ref{rank4}.

When all valuations of the quiver $Q$ are equal it is known, see
\cite{caldrein}, that for an indecomposable exceptional representation $M$ we
have $|Gr_{\bf e}(M)|_{|\Fbb|\to1}=\chi_c(Gr_{\bf e}(M))$.  Thus we see that, when
$Q$ is an equally valued quiver, setting $|\Fbb|\to1$ in Conjecture~\ref{qccconj}
gives the Caldero-Chapoton formula proved in \cite{caldkell}.

While completing the final draft of this manuscript we learned from Bernhard Keller that Fan Qin \cite{qin} proved Conjecture~\ref{qccconj} for acyclic equally valued quivers.

\section*{Acknowledgments}

The author would like to thank his advisor Arkady Berenstein for his infinite patience and support.  Great thanks is also due to Bernhard Keller for pointing out the results of Fan Qin and correcting a reference in the proof of Corollary~\ref{acyclic}.

\section{Definitions and Notation}\label{def}

We begin this section with a recollection of some of the terminology related to quantum
cluster algebras.  Let $L$ be a lattice of rank $m$ and $\Lambda:L\times L\to
\QQ$ a skew-symmetric bilinear form and let $d$ be the least common multiple of all denominators appearing in the image of $\Lambda$.  Let $q$ be a formal variable and
consider the ring of integer Laurent polynomials $\ZZ[q^{\pm\frac{1}{2d}}]$.  Define the
\textit{based quantum torus} associated to the pair $(L,\Lambda)$ to be the
$\ZZ[q^{\pm\frac{1}{2d}}]$-algebra $\Tcal_{\Lambda,q}$ with distinguished
$\ZZ[q^{\pm\frac{1}{2d}}]$-basis $\{X^e: e\in L\}$ with multiplication given by
\[X^eX^f=q^{\Lambda(e,f)/2}X^{e+f}.\]  An easy computation shows that $\Tcal_{\Lambda,q}$ is
associative and the basis elements satisfy the following relations:
\[X^eX^f=q^{\Lambda(e,f)}X^fX^e,\ X^0=1,\ (X^e)^{-1}=X^{-e}.\]  As the based
quantum torus is an Ore domain, it is contained in its skew-field of fractions
$\Fcal_{\Lambda,q}$.

A \textit{toric frame} in $\Fcal_{\Lambda,q}$ is a map $M: \ZZ^m\to \Fcal_{\Lambda,q} \setminus \{0\}$
of the form \[M({\bf c})=\varphi(X^{\eta({\bf c})})\] where $\varphi$ is an
automorphism of $\Fcal_{\Lambda,q}$ and $\eta: \ZZ^m\to L$ is a lattice isomorphism.  The
$M({\bf c})$ form a $\ZZ[q^{\pm\frac{1}{2d}}]$-basis of the based quantum torus
$\Tcal_{\Lambda_M,q}:=\varphi(\Tcal_{\Lambda,q})$ which is an isomorphic copy of $\Tcal_{\Lambda,q}$ in $\Fcal_{\Lambda,q}$. 
The following equations hold in $\Tcal_{\Lambda_M,q}$:
\[M({\bf c})M({\bf d})=q^{\Lambda_M({\bf c},{\bf d})/2}M({\bf c}+{\bf d}),\
M({\bf c})M({\bf d})=q^{\Lambda_M({\bf c},{\bf d})}M({\bf d})M({\bf c}),\]
\[ M({\bf 0})=1,\ M({\bf c})^{-1}=M(-{\bf c}),\]
where $\Lambda_M$ is the skew-symmetric bilinear form on $\ZZ^m$ obtained from
the lattice isomorphism $\eta$.  Let $\Lambda_M$ also denote the skew-symmetric
$m\times m$ matrix defined by $\lambda_{ij}=\Lambda_M(\alpha_i,\alpha_j)$ where
$\alpha_1, \ldots, \alpha_m$ are the standard basis vectors in $\ZZ^m$.  Given a
toric frame $M$, let $X_i=M(\alpha_i)$.  Then we have
$$\Tcal_{\Lambda_M,q}=\ZZ[q^{\pm\frac{1}{2d}}]\langle X_1^{\pm 1}, \ldots,
X_m^{\pm1}:X_iX_j=q^{\lambda_{ij}}X_jX_i\rangle.$$  Using the relations above we
get for ${\bf c}\in\ZZ^m$ \[M({\bf c})=q^{\frac{1}{2}\sum_{i<j}
c_ic_j\lambda_{ji}}X_1^{c_1}X_2^{c_2}\cdots X_m^{c_m}=:X^{({\bf c})}.\]  

Let $\Lambda$ be an $m\times m$ skew-symmetric matrix and let $\tilde{B}$ be any
$m\times n$ matrix, $n\le m$.  We call the pair $(\Lambda, \tilde{B})$
\textit{compatible} if $\tilde{B}^T\Lambda=(D|0)$ is an $n\times m$ matrix with
$D$ a diagonal integer matrix with positive entries on the diagonal.  Throughout this paper we will assume $n=m$ with $m$ even and $\tilde{B}$ is invertible.  Otherwise, one may replace $\tilde{B}$ by the invertible $2n\times 2n$ matrix $\left(\begin{array}{cc}B & -I_n\\ I_n & R\end{array}\right)$, where $B$ is the principal part of $\tilde{B}$ and $I_n$ is the $n\times n$ identity matrix, such that $\tilde{B}^T\Lambda=\left(\begin{array}{cc}D & 0\\ 0 & D\end{array}\right)=:D\oplus D$, adjusting the coefficients and $\Lambda$ if necessary to ensure that $R$ is an integer matrix.  By restricting ourselves to mutations in directions $1,\ldots, n$ we will recover our original cluster algebra with principal coefficients.  

The pair $(M,\tilde{B})$ is called a \textit{quantum seed} if the pair $(\Lambda_M,
\tilde{B})$ is compatible. We now define the mutation of the quantum seed $(M,\tilde{B})$ in direction $k$ for $k\in[1,n]$.  Define the $m\times m$ matrix $E=(e_{ij})$ by 
\[e_{ij}=\begin{cases}
\delta_{ij} & \text{if $j\ne k$;}\\
-1 & \text{if $i=j=k$;}\\
[-b_{ik}]_+ & \text{if $i\ne j = k$}
\end{cases}
\]
where $[b]_+=max(0,b)$.
Let ${\bf c}=(c_1,\ldots,c_m)\in\ZZ^m$.  Define a map $M': \ZZ^m\to \Fcal_{\Lambda,q} \setminus \{0\}$ as follows:
\begin{equation}\label{eq:cl_exp}M'({\bf c})=\sum\limits_{p\ge0} {c_k \brack p}_{q^{d_k/2}} M(E{\bf c}+p{\bf b}^k),\end{equation}
where the vector ${\bf b}^k\in\ZZ^m$ is the $k$th column of $\tilde{B}$ and ${c_k \brack p}_{q^{d_k/2}}$ is given by equation \eqref{eq:sym-bin}.  
\begin{lemma}
$M'$ is a well-defined toric frame.
\end{lemma}
\begin{proof}
The case $c_k\ge0$ was shown in \cite[Section 4]{berzel}.  For $c_k<0$, this follows from an obvious extension of the argument in \cite[Section 4]{berzel} using 
the identity $\prod_{r=0}^{n-1}\frac{1}{1+q^rt}=\sum\limits_{r\ge0}q^{\half(n-1)r}{-n\brack r}_{q^{1/2}}t^r$ and the fact that $M'(-{\bf c})^{-1}$ is bar-invariant.
\end{proof}

Equation \eqref{eq:cl_exp} defines a birational isomorphism of based quantum tori $\mu_k:\Tcal_{\Lambda_{M'},q}\to \Tcal_{\Lambda_M,q}$ given by
$$\mu_k(X_i)=\begin{cases}X_i & \text{if $i\ne k$;}\\ M\left(\sum\limits_{\ell=1}^m[-b_{\ell k}]_+\alpha_\ell\right)+M\left(\sum\limits_{\ell=1}^m[b_{\ell k}]_+\alpha_\ell\right) & \text{if $i=k$.}\end{cases}$$  This map takes cluster variables in $\Tcal_{\Lambda_{M'},q}$ to cluster variables in $\Tcal_{\Lambda_M,q}$ and hence $\mu_k\Acal_q(\mu_kQ)=\Acal_q(Q)$.  So applying the isomorphism $\mu_k$ is the same as mutating the initial cluster of the cluster algebra in direction $k$.

Let $\tilde{B}'=\mu_k\tilde{B}=(b_{ij}')$ where 
$$b_{ij}'=
\begin{cases}
-b_{ij} & \text{if $i=k$ or $j=k$}\\
b_{ij}+sgn(b_{ik})[b_{ik}b_{kj}]_+ & \text{otherwise}
\end{cases}$$
where $[b]_+=max(0,b)$.
Then the quantum seed $(M',\tilde{B}')$ 
is defined to be the mutation of $(M,\tilde{B})$ in direction $k$, written $\mu_k(M,\tilde{B})$.  
Since $\tilde{B}$ was invertible, $\mu_k\tilde{B}$ is also invertible.  Suppose $\Lambda_M\tilde{B}=-D$.  Then $\Lambda_{M'}\mu_k\tilde{B}=-D$ and so $\Lambda_{M'}$ can be recovered 
from the valued quiver $(\mu_k\tilde{B},D)$.  Thus we will abuse notation further and call the invertible valued quiver $Q=(\tilde{B},D)$ 
a seed.  We will also denote $\mu_kQ=(\mu_k\tilde{B},D)$.  

A quantum seed $(M',\tilde{B}')$ is \emph{mutation
equivalent} to the seed $(M,\tilde{B})$ if there is a sequence of mutations taking one to the other, in this case write $(M',\tilde{B}')\sim(M,\tilde{B})$.
Let $\Xcal=\{M'(\alpha_i): (M',\tilde{B}')\sim(M,\tilde{B}), i\in[1,m]\}$.  The
elements of $\Xcal$ are called \textit{cluster variables}.  The quantum Laurent phenomenon (\cite[Corollary 5.2]{berzel}) states that $\Xcal$ is actually contained in $\Tcal_{\Lambda_M,q}$.
Since $\tilde{B}$ is invertible, the data of the compatible pair $(\Lambda_M,\tilde{B})$ is equivalent to the data of the invertible valued quiver $Q=(\tilde{B},-\Lambda_M\tilde{B})$.  Thus we will let $\Acal_q(Q)$ denote the $\ZZ[q^{\pm\frac{1}{2d}}]$-subalgebra of $\Tcal_{\Lambda_Q,q}:=\Tcal_{\Lambda_M,q}$ generated by $\Xcal$, called the \emph{quantum cluster algebra}.  

Let $Q$ be the valued quiver $(B,D)$, where $D=diag(d_1,\ldots,d_m)$.  
Let $\Fbb$ be a finite field.  We will define an $\Fbb$-species $A_Q$ and show that modules over $A_Q$ are the same as representations of $Q$.  Let $\delta_{ij}=gcd(d_i,d_j)$ and $\delta^{ij}=lcm(d_i,d_j)$.  Let $\bar{\Fbb}$ be an algebraic closure of $\Fbb$ and define $K_i=\Fbb_{d_i}$, the
degree $d_i$ extension of $\Fbb$ in $\bar{\Fbb}$.  Also denote $K_{ij}=\Fbb_{d_i}\cap \Fbb_{d_j}=\Fbb_{\delta_{ij}}$ and $K^{ij}=\Fbb_{\delta^{ij}}$.

Define $A_0=\prod_{i=1}^{m}
K_i$ and $A_1=\bigoplus_{b_{ij}>0} A_{ij}$ where $A_{ij}:=\Fbb_{d_ib_{ij}}$.
 We define a $K_i-K_j$-bimodule structure on $A_{ij}$ by setting
$A_{ij}=\Fbb^{r_{ij}}\otimes_{\Fbb} K^{ij}\cong \bigoplus_{k=1}^{r_{ij}} 
K^{ij}$, where $r_{ij}=\frac{d_ib_{ij}}{\delta^{ij}}=gcd(|b_{ij}|,|b_{ji}|)$. 
The following easy lemma shows that this gives such a structure.

\begin{lemma}
$K_i\otimes_{K_{ij}} K_j$ is a field isomorphic to $K^{ij}$.  In particular,
$K^{ij}$ is a $K_i-K_j$-bimodule.
\end{lemma}

Now define $A_Q=T(A_0,A_1)$ the tensor algebra of $A_1$ over $A_0$.  A
module $X$ over $A_Q$ is given by a $K_i$-vector space $X_i$ for each
vertex $i$ and a $K_j$-linear map $\theta_{ij}: X_i\otimes_{K_i} A_{ij} \to X_j$
whenever $b_{ij}>0$, see \cite{hub}.  A morphism of $A_Q$-modules $f:X\to Y$ is a collection $\{f_i\}_{i\in[1,m]}$ with $f_i:X_i\to Y_i$ a $K_i$-linear map, such that $\theta_{ij}^Y(f_i\otimes id)=f_j\theta_{ij}^X$.  

Consider the following natural isomorphisms:
\[Hom_{K_j}(X_i\otimes_{K_i}A_{ij},X_j)\cong\Fbb^{r_{ij}}\otimes_{\Fbb}
Hom_{K_j}(X_i\otimes_{K_i}K^{ij},X_j),\]
\[X_i\otimes_{K_i}K^{ij}=X_i\otimes_{K_i}(K_i\otimes_{K_{ij}}K_j)\cong
X_i\otimes_{K_{ij}}K_j,\]
\[Hom_{K_j}(X_i\otimes_{K_{ij}}K_j,X_j)\cong Hom_{K_{ij}}(X_i,X_j).\]
Combining these we obtain a natural isomorphism 

\begin{equation}\label{eq:nat}
Hom_{K_j}(X_i\otimes_{K_i}A_{ij},X_j)\cong\Fbb^{r_{ij}}\otimes_{\Fbb}
Hom_{K_j}(X_i\otimes_{K_{ij}}K_j,X_j)\cong\bigoplus_{\ell=1}^{r_{ij}}Hom_{K_{ij}
}(X_i,X_j).
\end{equation}

These natural isomorphisms define inverse equivalences of categories $$F:mod~A_Q\leftrightarrow rep~Q:F^{-1}.$$  Indeed one can easily check that the commuting squares defining morphisms of modules and representations are compatible under the natural isomorphism \ref{eq:nat}. Thus the categories $mod~ A_Q$ and $rep~Q$ are equivalent and we can apply all results concerning $\Fbb$-species to valued representations.

Let $\Qcal$ denote the Grothendieck group of $rep~Q$.
For a valued representation $V$ denote by $[V]\in\Qcal$ the isomorphism class of $V$.  Clearly, $[V]=\sum_{i\in Q} (\dim_{K_i}V_i)\alpha_i$ where $\alpha_i=[S_i]$.  For objects $V,W\in
rep~Q$ define the Euler form \[\langle V,W\rangle=\dim_{\Fbb}
\Hom(V,W)-\dim_{\Fbb} \Ext^1(V,W),\] where $\Hom$ and $\Ext$ are computed in
$rep~Q$.  It is known, see \cite{schiffmann} for example, that
$\langle V,W\rangle$ only depends on the classes of $V$ and $W$ in
$\Qcal$.

Let $V\in rep~Q$ and define $Gr_{{\bf e}}(V)$ to be the set of
subrepresentations $W$ of $V$ with $[W]={\bf e}$.  For ${\bf e} \in \Qcal$
define vectors ${}^*{\bf e},{\bf e}^*\in \ZZ^n$ as in the introduction.  Define
$X_V\in\Tcal_{\Lambda_Q,|\Fbb|}$ by
\begin{equation}\label{qcc-formula}
X_V=\sum_{{\bf e}} |\Fbb|^{-\half\langle{\bf e},[V]-{\bf e}\rangle}|Gr_{{\bf
e}}(V)|X^{(B{\bf e}-{}^*[V])}.
\end{equation}
 
Let $rep~Q\langle i\rangle$ denote the full subcategory of
$rep~Q$ of all representations of $Q$ which do not contain
$S_i$ as a direct summand.  In \cite{dlab} it is shown that the reflection
functors $$\Sbb_i^-: rep~Q\leftrightarrow
rep~\mu_iQ:\Sbb_i^+$$ restrict to inverse equivalences
of categories $$\Sbb_i^-: rep~Q\langle
i\rangle\leftrightarrow rep~\mu_iQ\langle
i\rangle:\Sbb_i^+.$$  Since it will be clear from context which to use we will drop the ${}^\pm$ and simply denote both functors by $\Sbb_i$. See \cite{dlab} for precise definitions of these
functors.  We will use the following result proved in \cite{dlab}.

\begin{lemma}\cite[Proposition 2.1]{dlab}\label{sigma-prop}\mbox{}
For $X\in rep~Q\langle i\rangle$ we have $[\Sbb_i X]=\sigma_i([X])$ and $\Sbb_i^2X=X$.
\end{lemma}

Our main tool will be the following powerful result proved in section~\ref{pf:mutations}.

\begin{theorem}\label{tool}
For any $N\in rep~Q\langle i\rangle$, $\mu_iX_N=X_{\Sbb_iN}$.  In particular, if $X_N$ is a cluster variable in $\Acal_q(Q)$ then $X_{\Sbb_iN}$ is a cluster variable in $\Acal_q(\mu_iQ)$. 
\end{theorem}

In order to state our main theorem we introduce some new notation for describing a cluster variable.  Define $X_{[a_0]}^Q:=X^{(\alpha_{a_0})}$ in $\Acal_q(Q)$, so
the ordered tuple $(X_{[1]},\ldots,X_{[n]})$ forms the initial cluster of $\Acal_q(Q)$.  Now
recursively define
$$X_{[a_0;a_1,a_2,\ldots,a_r]}^Q=\mu_{a_r}X_{[a_0;a_1,a_2,\ldots,a_{r-1}]}^{\mu_{a_r}Q}$$ 
where $X_{[a_0;a_1,a_2,\ldots,a_{r-1}]}^{\mu_{a_r}Q}$ is in $\Acal_q(\mu_{a_r}Q)$ and the birational isomorphism $\mu_{a_r}$ pulls it back to $\Acal_q(Q)$.  Alternatively one could start with the initial ordered seed $({\bf X},Q)$ and mutate in direction $a_1$, then $a_2$, etc. to obtain the ordered seed $\mu_{a_r}\cdots\mu_{a_1}({\bf X},Q)$ then $X_{[a_0;a_1,a_2,\ldots,a_r]}^Q$ is the $a_0{}^{th}$ cluster variable in this seed.  When it is clear from context we will drop the $Q$ from the notation.

Here are some simple observations that follow from this notation:
\begin{enumerate}
\item If $a_i=a_{i+1}$ for some $i>0$, then
$X_{[a_0;a_1,a_2,\ldots,a_r]}^Q=X_{[a_0;a_1,\ldots,a_{i-1},a_{i+2},\ldots,a_r]}^Q$.
\item If $a_0\neq a_r$, then
$X_{[a_0;a_1,a_2,\ldots,a_r]}^Q=X_{[a_0;a_1,a_2,\ldots,a_{r-1}]}^Q$.
\item If we mutate the seed
$(X_{[1;a_1,\ldots,a_r]}^Q,\ldots,X_{[n;a_1,\ldots,a_r]}^Q,Q')$ in direction $t$ we get
the new seed $(X_{[1;a_1,\ldots,a_r,t]}^Q,\ldots,X_{[n;a_1,\ldots,a_r,t]}^Q,\mu_tQ')$.
\item If we start with $X_{[a_0;a_1,a_2,\ldots,a_{r-1}]}^Q$ and mutate the \emph{initial} seed in direction $t$ then we get $X_{[a_0;t,a_1,a_2,\ldots,a_{r-1}]}^{\mu_tQ}$.
\end{enumerate}

The content of Theorem \ref{main} is contained in the following Theorem and Corollary.

\begin{theorem}\label{th:mutations}
Suppose $k_1$, $k_2$, \ldots, $k_{r+1}$ is an admissible sequence of vertices in $Q$.  Let $M\in rep~Q$ be the unique indecomposable representation of $Q$ with $[M]=\sigma_{k_1}\sigma_{k_2}\cdots\sigma_{k_r}(\alpha_{k_{r+1}})$.  Then in the cluster algebra $\Acal_{|\Fbb|}(Q)$ we have $X_{[k_{r+1};k_1,k_2,\ldots,k_{r+1}]}^Q=X_M$.
\end{theorem}

We prove this in section~\ref{pf:mutations}.  We will call a quiver $Q$ \emph{almost acyclic} if there exists $i$ so that the valued quiver $\mu_iQ$ is acyclic.

\begin{corollary}\label{acyclic}
Suppose the valued quiver $Q$ is acyclic.  Then each cluster variable of $\Acal_{|\Fbb|}(Q)$ in an almost acyclic cluster is of the form $X_N$ for some indecomposable $N\in rep~Q$.
\end{corollary}
\begin{proof}
In \cite[Corollary 4]{caldkell}, the authors show that all acyclic clusters are connected by sink and source mutations.  The result follows from Theorem~\ref{th:mutations}.
\end{proof}

Theorem~\ref{rank2} immediately follows.

\begin{proof}[Proof of Theorem~\ref{rank2}]
For rank 2 cluster algebras, the valued quiver associated to any cluster is acyclic.  So the result follows from Corollary~\ref{acyclic}.
\end{proof}

We also get Theorem~\ref{cor:fintype}.

\begin{proof}[Proof of Theorem~\ref{cor:fintype}]
We will use the concepts from \cite{dlab}.  
Let $k_1, k_2,\ldots, k_n$ be an admissible ordering of $Q$ and $C=\Sbb_{k_1}\Sbb_{k_2}\cdots\Sbb_{k_n}$ be the corresponding Coxeter functor.  Let $P_{k_t}=\Sbb_{k_1}\Sbb_{k_2}\cdots\Sbb_{k_{t-1}}S_{k_t}$ where $S_{k_t}\in rep~\mu_{k_t}\mu_{k_{t+1}}\cdots\mu_{k_{t_n}}Q$ is the simple representation associated to vertex $k_t$.  By \cite[Propositions 1.9 and 2.6]{dlab}, 
every indecomposable representation of $Q$ is of the form $C^rP_{k_t}$ for some $t$ and $1\le r\le a_t$ where $a_t$ is the largest integer for which all such $C^rP_{k_t}$ are nonzero.   

There is a one-to-one correspondence between cluster variables of $\Acal_{|\Fbb|}(Q)$ and positive roots in the root system $\Phi_Q$ and thus a one-to-one correspondence between cluster variables and indecomposable representations of $Q$.  From Theorem~\ref{th:mutations} and the definition of $P_{k_t}$ we see that $X_{P_{k_t}}$ is a cluster variable in $\Acal_{|\Fbb|}(Q)$.

The result follows from the proof of Theorem~\ref{th:mutations} if $X_{P_{k_t}}$.
\end{proof}

\begin{proof}[Proof of Proposition~\ref{prop:compatible}]
Write $\underline{\mu_k}:\Fcal_{D\mu_kB,q}\to\Fcal_{DB,q}$ for the mutation in direction $k$ from \cite[Section 3.3]{gonch} and $j': \Fcal_{D\mu_kB,q}\to \Fcal_{\Lambda_{\mu_kQ},q}$.  The following Lemma follows from the definitions.
\begin{lemma}
$\mu_kj'(F')=j\underline{\mu_k}(F')$ for any $F'\in\Fcal_{D\mu_kB,q}$.
\end{lemma}
This says that the mutation of the $\mu_kB$-compatible element $j'(F')$ is $B$-compatible.  Now equation~\eqref{eq:cl_exp} says that the mutation of any monomial in $\Fcal_{\Lambda_{\mu_kQ},q}$ is $B$-compatible.  Combining the last two statements completes the proof of Proposition~\ref{prop:compatible}\eqref{comp}.
Proposition~\ref{prop:compatible}\eqref{f-factor} follows from an easy computation using the fact that $\langle\alpha_i,\alpha_j\rangle=\Lambda({\bf b}^i,{}^*\alpha_j)$.
\end{proof}

\begin{proof}[Proof of Corollary~\ref{cor:fin-fpoly}]
Note that for finite types the Cartan counterpart of $B$ is always positive definite and so for any ${\bf e}\in\Qcal\setminus0$ we have $\langle{\bf e},{\bf e}\rangle>0$.  So this follows from Proposition~\ref{prop:compatible}\eqref{f-factor}, \cite[Theorem 6.1]{tran2}, \cite{qin}, and Section~\ref{g2}.
\end{proof}

\section{Examples}
Throughout this section we will let $\Fbb$ be the finite field with $q$ elements.  In what follows each cluster variable $X_{[a_0;a_1,a_2,\ldots,a_r]}$ will have $a_0=a_r$ so we will drop $a_0$ from the notation.

\subsection{Finite Type Rank 2 Quantum Cluster Algebras}\label{finitetype}

In what follows, we abbreviate
$X^{(a_1,a_2)}:=q^{-\frac{1}{2}a_1a_2\lambda_{12}}X_1^{a_1}X_2^{a_2}$, without
loss of generality we may assume $\lambda_{12}=1$.  Also note that there is no
significance to the numbering of the indecomposable representations.

\subsubsection{Type $A_2$}
We begin with the valued quiver $Q=(B,D)$ where $B=\left(\begin{array}{cc}0&1\\-1&0\end{array}\right)$ and $D=diag(1,1)$.
The Cartan counterpart of $B$ is $\left(\begin{array}{cc}2&-1\\-1&2\end{array}\right)$. 

\begin{lemma}[\cite{berzel}] The quantum cluster algebra 
$\Acal_q(Q)$ is of finite type and we have the following formulas for
the cluster variables in terms of the initial cluster $(X_1,X_2)$:
\begin{align*}
&X_3:=X_{[1]}=X^{(-1,1)}+X^{(-1,0)}\\
&X_4:=X_{[1,2]}=X^{(-1,0)}+X^{(0,-1)}+X^{(-1,-1)}\\
&X_5:=X_{[1,2,1]}=X^{(1,-1)}+X^{(0,-1)}\\
&X_6:=X_{[1,2,1,2]}=X^{(1,0)}=X_1\\
&X_7:=X_{[1,2,1,2,1]}=X^{(0,1)}=X_2.
\end{align*}
\end{lemma}
The indecomposable valued representations of $Q$ are $S_1=\Fbb\to0$
(dimension vector $\alpha_1$), $S_2=0\to \Fbb$ (dimension vector
$\alpha_2=s_1s_2(\alpha_1)$), and $I_1=\Fbb\stackrel{id}{\rightarrow}\Fbb$ (dimension
vector $\alpha_1+\alpha_2=s_1(\alpha_2)$).

The representation $S_1$ has unique subrepresentations with dimension vectors
$\alpha_1$ and $0$.  So we get:
\[X_{S_1}=X^{(-\alpha_1+\alpha_2)}+X^{(-\alpha_1)}=X_3.\]
The representation $I_1$ has unique subrepresentations with dimension vectors
$\alpha_1+\alpha_2$, $0$ and $\alpha_2$.  So we get:
\[X_{I_1}=X^{(-\alpha_1)}+X^{(-\alpha_2)}+X^{(-\alpha_1-\alpha_2)}=X_4.\]
The representation $S_2$ has unique subrepresentations with dimension vectors
$0$ and $\alpha_2$.  So we get:
\[X_{S_2}=X^{(\alpha_1-\alpha_2)}+X^{(-\alpha_2)}=X_5.\]

\subsubsection{Type $C_2$}
(We get type $B_2$ by dualizing all representations.)  We begin with the valued quiver 
$Q=(B,D)$ where $B=\left(\begin{array}{cc}0&2\\-1&0\end{array}\right)$ and $D=diag(1,2)$.
Define $K_1=\Fbb=:k$ and $K_2=\Fbb_{2}=:K$.  The Cartan counterpart of $B$ is
$\left(\begin{array}{cc}2&-2\\-1&2\end{array}\right)$. 

\begin{lemma}[\cite{berzel}] The quantum cluster algebra
$\Acal_q(Q)$ is of finite type and we have the following formulas for
the cluster variables in terms of the initial cluster $(X_1,X_2)$:
\begin{align*}
&X_3:=X_{[1]}=X^{(-1,1)}+X^{(-1,0)}\\
&X_4:=X_{[1,2]}=X^{(-2,1)}+(q^{1/2}+q^{-1/2})X^{(-2,0)}+X^{(0,-1)}+X^{(-2,-1)}
\\
&X_5:=X_{[1,2,1]}=X^{(-1,0)}+X^{(1,-1)}+X^{(-1,-1)}\\
&X_6:=X_{[1,2,1,2]}=X^{(2,-1)}+X^{(0,-1)}\\
&X_7:=X_{[1,2,1,2,1]}=X^{(1,0)}=X_1\\
&X_8:=X_{[1,2,1,2,1,2]}=X^{(0,1)}=X_2.
\end{align*}
\end{lemma}
The indecomposable valued representations of $Q$ are $S_1=k\to0$
(dimension vector $\alpha_1$), $S_2=0\to K$ (dimension vector
$\alpha_2=s_1s_2s_1(\alpha_2)$), $I_1=k\stackrel{\iota}{\rightarrow}K$
(dimension vector $\alpha_1+\alpha_2=s_1s_2(\alpha_1)$) where $\iota$ is the
inclusion map, and $I_2=k^2\stackrel{\sigma}{\rightarrow}K$ (dimension vector
$2\alpha_1+\alpha_2=s_1(\alpha_2)$) where $\sigma$ identifies $K$ as a
2-dimensional vector space over $k$.

The representation $S_1$ has unique subrepresentations with dimension vectors
$\alpha_1$ and $0$.  So we get:
\[X_{S_1}=X^{(-\alpha_1+\alpha_2)}+X^{(-\alpha_1)}=X_3.\]
The representation $I_2$ has unique subrepresentations with dimension vectors
$2\alpha_1+\alpha_2$, $0$ and $\alpha_2$, and it has $1+q$ subrepresentations
with dimension vector $\alpha_1+\alpha_2$.  So we get:
\[X_{I_1}=X^{(-2\alpha_1+\alpha_2)}+X^{(-\alpha_2)}+X^{(-2\alpha_1-\alpha_2)}
+(q^{1/2}+q^{-1/2})X^{(-2\alpha_1)}=X_4.\]
The representation $I_1$ has unique subrepresentations with dimension vectors
$\alpha_1+\alpha_2$, $0$ and $\alpha_2$.  So we get:
\[X_{I_1}=X^{(-\alpha_1)}+X^{(\alpha_1-\alpha_2)}+X^{(-\alpha_1-\alpha_2)}=X_5.\]
The representation $S_2$ has unique subrepresentations with dimension vectors
$0$ and $\alpha_2$.  So we get:
\[X_{S_2}=X^{(2\alpha_1-\alpha_2)}+X^{(-\alpha_2)}=X_6.\]

\subsubsection{Type $G_2$}\label{g2}
We begin with the valued quiver 
$Q=(B,D)$ where $B=\left(\begin{array}{cc}0&3\\-1&0\end{array}\right)$ and $D=diag(1,3)$.
Define $K_1=\Fbb=:k$ and $K_2=\Fbb_{3}=:K$.  The Cartan counterpart of $B$ is
$\left(\begin{array}{cc}2&-3\\-1&2\end{array}\right)$. 

\begin{lemma}[\cite{berzel}] The quantum cluster algebra
$\Acal_q(Q)$ is of finite type and we have the following formulas for
the cluster variables in terms of the initial cluster $(X_1,X_2)$:
\begin{align*}
&X_3:=X_{[1]}=X^{(-1,1)}+X^{(-1,0)}\\
&X_4:=X_{[1,2]}=X^{(-3,2)}+(q+1+q^{-1})X^{(-3,1)}+(q+1+q^{-1})X^{(-3,0)}\\
&\ \ \ +X^{(0,-1)}+X^{(-3,-1)}\\
&X_5:=X_{[1,2,1]}=X^{(-2,1)}+(q^{1/2}+q^{-1/2})X^{(-2,0)}+X^{(1,-1)}+X^{(-2,
-1)}\\
&X_6:=X_{[1,2,1,2]}=X^{(-3,1)}+(q+1+q^{-1})X^{(-3,0)}+(q+1+q^{-1})X^{(0,-1)}\\
&\ \ \
+(q+1+q^{-1})X^{(-3,-1)}+X^{(3,-2)}+(q^{3/2}+q^{-3/2})X^{(0,-2)}+X^{(-3,-2)}\\
&X_7:=X_{[1,2,1,2,1]}=X^{(-1,0)}+X^{(2,-1)}+X^{(-1,-1)}\\
&X_8:=X_{[1,2,1,2,1,2]}=X^{(3,-1)}+X^{(0,-1)}\\
&X_9:=X_{[1,2,1,2,1,2,1]}=X^{(1,0)}=X_1\\
&X_{10}:=X_{1,2,1,2,1,2,1,2]}=X^{(0,1)}=X_2.
\end{align*}
\end{lemma}
The indecomposable valued representations of $Q$ are $S_1=k\to0$ (dimension
vector $\alpha_1$), $S_2=0\to K$ (dimension vector
$\alpha_2=s_1s_2s_1s_2s_1(\alpha_2)$), $I_1=k\stackrel{\iota}{\rightarrow}K$
(dimension vector $\alpha_1+\alpha_2=s_1s_2s_1s_2(\alpha_1)$) where $\iota$ is
the inclusion map, $I_2=k^2\stackrel{\sigma}{\rightarrow}K$ (dimension vector
$2\alpha_1+\alpha_2=s_1s_2(\alpha_1)$) where $\sigma$ identifies a 2-dimensional
$k$-subspace of $K$, $I_3=k^3\stackrel{\tau}{\rightarrow}K$ (dimension vector
$3\alpha_1+\alpha_2=s_1(\alpha_2)$) where $\tau$ identifies $K$ as a
3-dimensional vector space over $k$, and
$I_4=k^3\stackrel{\omega}{\rightarrow}K^2$ (dimension vector
$3\alpha_1+2\alpha_2=s_1s_2s_1(\alpha_2)$) where
$\omega=(\iota_1,\iota_2,\Delta_\iota)$ with $\iota_i$ the inclusion map to the
$i^{\text{th}}$ factor of $K^2$ and $\Delta_\iota$ the diagonal inclusion map.

The representation $S_1$ has unique subrepresentations with dimension vectors
$\alpha_1$ and $0$.  So we get:
\[X_{S_1}=X^{(-\alpha_1+\alpha_2)}+X^{(-\alpha_1)}=X_3.\]
The representation $I_3$ has unique subrepresentations with dimension vectors
$3\alpha_1+\alpha_2$, $0$ and $\alpha_2$, and it has $1+q+q^2$
subrepresentations with dimension vectors $2\alpha_1+\alpha_2$ and
$\alpha_1+\alpha_2$.  So we get:
\begin{align*}&X_{I_3}=X^{(-3\alpha_1+2\alpha_2)}+X^{(-\alpha_2)}+X^{
(-3\alpha_1-\alpha_2)}+(q+1+q^{-1})X^{(-3\alpha_1+\alpha_2)}\\ &\ \ \
+(q+1+q^{-1})X^{(-3\alpha_1)}=X_4.\end{align*}
The representation $I_2$ has unique subrepresentations with dimension vectors
$2\alpha_1+\alpha_2$, $0$ and $\alpha_2$, and it has $1+q$ subrepresentations
with dimension vector $\alpha_1+\alpha_2$.  So we get:
\[X_{I_2}=X^{(-2\alpha_1+\alpha_2)}+X^{(\alpha_1-\alpha_2)}+X^{
(-2\alpha_1-\alpha_2)}+(q^{1/2}+q^{-1/2})X^{(-2\alpha_1)}=X_5.\]
The representation $I_4$ has unique subrepresentations with dimension vectors
$3\alpha_1+2\alpha_2$, $0$ and $2\alpha_2$, it has $1+q+q^2$ subrepresentations
with dimension vectors $2\alpha_1+2\alpha_2$, $\alpha_1+2\alpha_2$, and
$\alpha_1+\alpha_2$ (choosing a 1-dimensional $k$-subspace of $k^3$ forces the
1-dimensional $K$-subspace of $K^2$), and it has $1+q^3$ subrepresentations with
dimension vector $\alpha_2$.  So we get:
\begin{align*}&X_{I_4}=X^{(-3\alpha_1+\alpha_2)}+X^{(3\alpha_1-2\alpha_2)}+X^{
(-3\alpha_1-2\alpha_2)}\\ &\ \ \
+(q+1+q^{-1})X^{(-3\alpha_1)}+(q+1+q^{-1})X^{(-3\alpha_1-\alpha_2)}\\ &\ \ \
+(q+1+q^{-1})X^{(-\alpha_2)}+(q^{3/2}+q^{-3/2})X^{(-2\alpha_2)}=X_6.\end{align*}
The representation $I_1$ has unique subrepresentations with dimension vectors
$\alpha_1+\alpha_2$, $0$ and $\alpha_2$ so we get:
\[X_{I_1}=X^{(-\alpha_1)}+X^{(2\alpha_1-\alpha_2)}+X^{(-\alpha_1-\alpha_2)}
=X_7.\]
The representation $S_2$ has unique subrepresentations with dimension vectors
$0$ and $\alpha_2$ so we get:
\[X_{S_2}=X^{(3\alpha_1-\alpha_2)}+X^{(-\alpha_2)}=X_8.\]
\\

\subsection{A Rank 4 Example}\label{rank4}

In this section we will work in the quantum cluster algebra $\Acal_q(Q)$ where $Q=(B,D)$ is the acyclic valued quiver with
$$B=\left(\begin{array}{cccc} 0&2&0&0\\-2&0&2&0\\0&-2&0&2\\0&0&-2&0\end{array}\right)$$ and $D=diag(1,1,1,1)$.
One can easily compute the following cluster variables of $\Acal_q(Q)$ each living in a \emph{cyclic} cluster:
\begin{align*}
X_{[2;2]}&=X^{(0,-1,2,0)}+X^{(2,-1,0,0)}\\
X_{[3;2,3]}&=X^{(0,-2,3,2)}+(q^{1/2}+q^{-1/2})X^{(2,-2,1,2)}+X^{(4,-2,-1,2)}+X^{(4,0,-1,0)}\\
X_{[4;2,3,4]}&=X^{(0,-4,6,3)}+(q^{3/2}+q^{1/2}+q^{-1/2}+q^{-3/2})X^{(2,-4,4,3)}+(q^{2}+q+2+q^{-1}+q^{-2})X^{(4,-4,2,3)}\\
& \quad +(q^{3/2}+q^{1/2}+q^{-1/2}+q^{-3/2})X^{(6,-4,0,3)}+X^{(8,-4,-2,3)}+(q^{1/2}+q^{-1/2})X^{(4,-2,2,1)}\\
& \quad +(q+2+q^{-1})X^{(6,-2,0,1)}+(q^{1/2}+q^{-1/2})X^{(8,-2,-2,1)}+X^{(8,0,0,-1)}+X^{(8,0,-2,-1)}
\end{align*}

We verify Conjecture~\ref{qccconj} for these cluster variables.  
\begin{lemma}
Conjecture~\ref{qccconj} holds for $X_{[2;2]}$.
\end{lemma}
\begin{proof}
By Lemma~\ref{simple} we have $X_{[2;2]}=X_{S_2}$.  
\end{proof}
Let $I_1$ be the unique indecomposable representation of $Q$ with dimension vector $\sigma_2(\alpha_3)=2\alpha_2+\alpha_3$.  
\begin{lemma}
Conjecture~\ref{qccconj} holds for $X_{[3;2,3]}$.
\end{lemma}
\begin{proof}
The following table shows how each term of $X_{I_1}$ arises, in particular we see that $X_{[3;2,3]}=X_{I_1}$: 

\begin{tabular}{|c|c|c|c|}
\hline 
\raisebox{-2px}{{\bf e}} & \raisebox{-2px}{$d_{\bf e}^{I_1}$} & \raisebox{-2px}{$|Gr_{{\bf e}}(I_1)|$} & \raisebox{-2px}{${}^*{\bf e}-{\bf e}^*-{}^*(2\alpha_2+\alpha_3)$} \\
[0.3ex] \hline 
\raisebox{-2px}{$2\alpha_2+\alpha_3$} & \raisebox{-2px}{0} & \raisebox{-2px}{1} & \raisebox{-2px}{$(0,-2,3,2)$}\\
[0.3ex] \hline 
\raisebox{-2px}{$\alpha_2+\alpha_3$} & \raisebox{-2px}{1} & \raisebox{-2px}{${2\choose1}_q$} & \raisebox{-2px}{$(2,-2,1,2)$} \\
[0.3ex] \hline 
\raisebox{-2px}{$\alpha_3$} & \raisebox{-2px}{0} & \raisebox{-2px}{1} & \raisebox{-2px}{$(4,-2,-1,2)$}\\
[0.3ex] \hline 
\raisebox{-2px}{0} & \raisebox{-2px}{0} & \raisebox{-2px}{1} & \raisebox{-2px}{$(4,0,-1,0)$}\\ [0.3ex] \hline
\end{tabular}\\
\end{proof}
Let $I_2$ be the unique indecomposable representation of $Q$ with dimension vector $\sigma_2\sigma_3(\alpha_4)=4\alpha_2+2\alpha_3+\alpha_4$.   
\begin{lemma}
Conjecture~\ref{qccconj} holds for $X_{[4;2,3,4]}$.
\end{lemma}
\begin{proof}
The following table shows how each term of $X_{I_2}$ arises, in particular we see that $X_{[4;2,3,4]}=X_{I_2}$:\\

\begin{tabular}{|c|c|c|c|}
\hline 
\raisebox{-2px}{{\bf e}} & \raisebox{-2px}{$d_{\bf e}^{I_2}$} & \raisebox{-2px}{$|Gr_{{\bf e}}(I_2)|$} & \raisebox{-2px}{${}^*{\bf e}-{\bf e}^*-{}^*(4\alpha_2+2\alpha_3+\alpha_4)$} \\
[0.3ex] \hline 
\raisebox{-2px}{$4\alpha_2+2\alpha_3+\alpha_4$} & \raisebox{-2px}{0} & \raisebox{-2px}{1} & \raisebox{-2px}{$(0,-4,6,3)$}\\
[0.3ex] \hline 
\raisebox{-2px}{$3\alpha_2+2\alpha_3+\alpha_4$} & \raisebox{-2px}{3} & \raisebox{-2px}{${4\choose3}_q$} & \raisebox{-2px}{$(2,-4,4,3)$} \\
[0.3ex] \hline 
\raisebox{-2px}{$2\alpha_2+2\alpha_3+\alpha_4$} & \raisebox{-2px}{4} & \raisebox{-2px}{${4\choose2}_q$} & \raisebox{-2px}{$(4,-4,2,3)$} \\
[0.3ex] \hline 
\raisebox{-2px}{$\alpha_2+2\alpha_3+\alpha_4$} & \raisebox{-2px}{3} & \raisebox{-2px}{${4\choose1}_q$} & \raisebox{-2px}{$(6,-4,0,3)$} \\
[0.3ex] \hline 
\raisebox{-2px}{$2\alpha_3+\alpha_4$} & \raisebox{-2px}{0} & \raisebox{-2px}{1} & \raisebox{-2px}{$(8,-4,-2,3)$} \\
[0.3ex] \hline 
\raisebox{-2px}{$2\alpha_2+\alpha_3+\alpha_4$} & \raisebox{-2px}{1} & \raisebox{-2px}{${2\choose1}_q$} & \raisebox{-2px}{$(4,-2,2,1)$} \\
[0.3ex] \hline 
\raisebox{-2px}{$\alpha_2+\alpha_3+\alpha_4$} & \raisebox{-2px}{2} & \raisebox{-2px}{${2\choose1}_q{2\choose1}_{q^2}$} & \raisebox{-2px}{$(6,-2,0,1)$} \\
[0.3ex] \hline 
\raisebox{-2px}{$\alpha_3+\alpha_4$} & \raisebox{-2px}{1} & \raisebox{-2px}{${2\choose1}_q$} & \raisebox{-2px}{$(8,-2,-2,1)$} \\
[0.3ex] \hline 
\raisebox{-2px}{$\alpha_4$} & \raisebox{-2px}{0} & \raisebox{-2px}{1} & \raisebox{-2px}{$(8,0,-2,-1)$}\\
[0.3ex] \hline 
\raisebox{-2px}{$0$} & \raisebox{-2px}{0} & \raisebox{-2px}{1} & \raisebox{-2px}{$(8,0,0,-1)$}\\ [0.3ex] \hline
\end{tabular}\\
\end{proof}
\begin{remark}
The above computations are easily generalized to any linearly ordered rank 4 valued quiver. 
\end{remark}

\subsection{Type $A_n$}
In what follows we will work with ordinary quivers.  These can be considered as valued quivers
by assigning valuation 1 to each vertex.

In \cite{ca2}, Fomin and Zelevinsky show that the cluster algebras of type $A_n$
can be recovered from triangulations of the $(n+3)$-gon: the clusters are in
one-to-one correspondence with the triangulations.  In \cite{schiffler},
Schiffler gives a combinatorial description of the expansion of an arbitrary
cluster variable in terms of paths in a triangulation.  In this section we show
that this combinatorial description carries over to quantum cluster variables.

We begin by recalling some notions from \cite{schiffler}.
Let $P$ be an $(n+3)$-gon, with vertices labeled $v_0,v_1,\ldots, v_{n+2}$.  A
diagonal $D_{a,b}$ in $P$ is a line segment connecting two non-adjacent vertices
$a$ and $b$.  Two diagonals cross if they intersect in the interior of $P$ and a
triangulation of $P$ is a maximal set of non-crossing diagonals.  Note that each
triangulation contains exactly $n$ diagonals, label them $T_1,T_2,\ldots, T_n$
(these correspond to cluster variables), and $n+3$ boundary edges, label them
$T_{n+1},\ldots,T_{2n+3}$ (these correspond to coefficients).  We construct a
skew-symmetric $n\times n$ matrix $B_T$ from a triangulation $T$ as follows:
\begin{itemize}
\item For each $i$, $b_{ii}=0$.
\item Suppose diagonals $T_i\ne T_j$ bound the same triangle in $T$.  Then
$b_{ij}=1$ (respectively $b_{ij}=-1$) if the sense of rotation from $T_i$ to
$T_j$ is counterclockwise (respectively clockwise).
\item If $T_i$ and $T_j$ do not bound the same triangle in $T$, then $b_{ij}=0$.
\end{itemize}
Note that this process can be reversed: starting from a matrix $B$ we can
construct a triangulation of $P$. 

Let $T$ be a triangulation of $P$ and let $D_{a,b}$ be a diagonal of $P$.
\begin{definition}
A $T$-path $\rho$ from $a$ to $b$ is a sequence \[\rho=(a_0, a_1, \ldots
a_{\ell(\rho)}: i_1,i_2,\ldots,i_{\ell(\rho)})\] such that
\begin{itemize}
\item[(T1)] $a=a_0, a_1, \ldots, a_{\ell(\rho)}=b$ are vertices of $P$
\item[(T2)] $i_k\in\{1,2,\ldots, 2n+3\}$ such that $T_{i_k}$ connects the
vertices $a_{i_k-1}$ and $a_{i_k}$ for each $k$
\item[(T3)] $i_j\ne i_k$ if $j\ne k$
\item[(T4)] $\ell(\rho)$ is odd
\item[(T5)] $T_{i_k}$ crosses $D_{a,b}$ if $k$ is even
\item[(T6)] If $j<k$ and both $T_{i_j}$ and $T_{i_k}$ cross $D_{a,b}$, then the
crossing point of $T_{i_j}$ with $D_{a,b}$ is closer to the vertex $a$ than the
crossing point of $T_{i_k}$ with $D_{a,b}$.
\end{itemize}
\end{definition}

\begin{definition}
Let $\Pcal_T(a,b)$ denote the set of all $T$-paths from $a$ to $b$.
\end{definition}

Let $Q$ be the valued quiver $(B,diag(1,\ldots,1))$ and $\Acal_q(Q)$ be the associated quantum cluster algebra.  

Let $\{\alpha_i\}$ be the standard set of generators in $\ZZ^n$.  For a $T$-path
$\rho$ define 
\[\overline{\rho}=\sum_{k \text{ odd}}\alpha_{i_k}-\sum_{k \text{
even}}\alpha_{i_k}.\]

We have the following quantum analogue of the main theorem of \cite{schiffler}.

\begin{theorem}
Let $a$ and $b$ be two non-adjacent vertices of $P$, let $D=D_{a,b}$ be the
diagonal of $P$ connecting $a$ and $b$, and let $X_D$ be the corresponding
cluster variable in $\Acal_q(Q)$.  Then 
\[X_D=\sum_{\rho\in\Pcal_T(a,b)}X^{(\bar{\rho})}.\]
\end{theorem}
\begin{proof}
This is an obvious adaptation of the proof of \cite[Theorem 1.2]{schiffler}.
\end{proof}

\begin{remark}
Let $Q$ be the quiver associated to the matrix $B$.  By Corollary 1.4 we know that there exists $V_D\in rep~Q$ so that $X_D=X_{V_D}$. This correspondence between elements of $\Pcal_T(a,b)$ and subrepresentations of $V_D$ can easily be made explicit in the case of a linearly ordered quiver or an alternating quiver of type $A_n$.
\end{remark}

\section{Proof of Proposition \ref{th:XZ-formula}}

In this section we will again let $\Fbb$ be the finite field with $q$ elements.

Let $\Acal=\Acal_q(Q)$ where $Q=(B,D)$ with $B=\left(\begin{array}{cc}0&2\\-2&0\end{array}\right)$ and $D=diag(2,2)$.
This is the Kronecker quiver $\xymatrix{\circ \ar[r]^2 & \circ}$, with two
arrows, and $K_1=K_2=\Fbb_{2}$.
$\Acal$ is a subring of the skew-field $$\QQ(q^{1/2})\langle X_1,
X_2:X_1X_2=qX_2X_1\rangle$$ generated by the
elements $X_m$ for $m \in \ZZ$ satisfying the recurrence relations
\begin{equation}
\label{eq:X-recursion}
X_{m-1} X_{m+1} = qX_m^2+1 \quad (m \in \ZZ)\ .
\end{equation}

We use the results of \cite{sza} and Theorem~\ref{rank2} to prove closed formulas for the Laurent polynomial expressions of the elements
$X_m$.

The Kronecker quiver has two classes of non-regular indecomposable representations: the
preprojective representations $P_n$ with dimension vector $(n,n+1)$ and the
postinjective representations $I_n$ with dimension vector $(n+1,n)$ (note that
the arrows of our Kronecker quiver are reversed from those in \cite{sza} and we
are using degree 2 field extensions of $\Fbb_q$).  The preprojectives and
postinjectives are uniquely determined (up to isomorphism) by their dimension
vectors.  Define for $n,k\in\ZZ$, $k\ge0$, quantum binomial coefficients
${n\choose
k}_q=\frac{(q^n-1)(q^{n-1}-1)\cdots(q^{n-r+1}-1)}{(q^r-1)\cdots(q-1)}$ and take
${t\choose0}_q=1$ for any $t\in\ZZ$.  We have the following theorem proved in \cite{sza}:
\begin{theorem}\cite[Theorem 4.1, 4.3]{sza}
For $n\geq 0$,
\begin{enumerate}
\item $|Gr_{(a,b)}(P_n)|={n+1-a\choose n+1-b}_{q^2}{b-1\choose a}_{q^2}$
\item $|Gr_{(a,b)}(I_n)|={n-a \choose n-b}_{q^2}{b+1\choose a}_{q^2}$.
\end{enumerate}
\end{theorem}

Notice that for $a\ge b$ there are no
subrepresentations of $P_n$ with dimension vector $(a,b)$ unless $a=0$ and
$b=0$.  For ${\bf e}=(0,0)$ we get $|Gr_{{\bf e}}(P_n)|=1$ and ${}^*{\bf
e}-{\bf e}^*-{}^*[P_n]=(n+2,-n-1)$.  So ${\bf e}=(0,0)$ gives the floating
term $X^{(n+2,-n-1)}$ in equation \eqref{eq:X-n-formula}.  All remaining
dimension vectors of subrepresentations of $P_n$ are of the form $(a,b)$ with
$a< b\leq n+1$.  Set $p=a$ and $r=n+1-b$, so we get $r+p=a+n+1-b\leq n$.  For ${\bf
e}=(a,b)$ we get the summand 
\begin{align*}
&q^{-\half d_{\bf e}^{P_n}}|Gr_{(a,b)}(P_n)|X^{(n+2-2b,2a-n-1)}\\
&=q^{(a-b)(n+1-b)}{n+1-a\choose n+1-b}_{q^2}q^{(a-b+1)a}{b-1\choose
a}_{q^2}X^{(n+2-2b,2a-n-1)}\\
&\textstyle ={n+1-a\brack n+1-b}_q{b-1\brack a}_qX^{(n+2-2b,2a-n-1)}\\
&\textstyle ={n+1-p\brack r}_q{n-r\brack p}_qX^{(2r-n,2p-n-1)}.
\end{align*}

Now applying Theorem~\ref{rank2} we get for $n\in\ZZ_{\ge0}$
\begin{align}\label{eq:n}
X_{-n}&=X_{P_n}=\sum_{{\bf e}} q^{-\half d_{{\bf e}}^{P_n}}|Gr_{{\bf e}}(P_n)|
X^{(n+2-2e_2,2e_1-n-1)}\\
&\nonumber=X^{(n+2,-n-1)} + \sum_{p + r \leq
n} {n-r \brack p}_q{n+1-p \brack r}_q X^{(2r-n,2p-n-1)}.
\end{align}

Notice that for $a>b$ there are no
subrepresentations of $I_n$ with dimension vector $(a,b)$ unless $a=n+1$ and
$b=n$.  For ${\bf e}=(n+1,n)$ we get $|Gr_{{\bf e}}(I_n)|=1$ and ${}^*{\bf
e}-{\bf e}^*-{}^*[I_n]=(-n-1,n+2)$.  So ${\bf e}=(n+1,n)$ gives the floating
term $X^{(-n-1,n+2)}$ in equation \eqref{eq:Xn-formula}.  All remaining
dimension vectors of subrepresentations of $I_n$ are of the form $(a,b)$ with
$a\leq b\leq n$.  Set $r=a$ and $p=n-b$, so we get $r+p=a+n-b\leq n$.  For ${\bf
e}=(a,b)$ we get the summand 
\begin{align*}
&q^{-\half d_{\bf e}^{I_n}}|Gr_{(a,b)}(I_n)|X^{(n-1-2b,2a-n)}\\
&=q^{(a-b)(n-b)}{n-a\choose n-b}_{q^2}q^{(a-b-1)a}{b+1\choose
a}_{q^2}X^{(n-1-2b,2a-n)}\\
&\textstyle ={n-a\brack n-b}_q{b+1\brack a}_qX^{(n-1-2b,2a-n)}\\
&\textstyle ={n-r\brack p}_q{n+1-p\brack r}_qX^{(2p-n-1,2r-n)}.
\end{align*}

Again applying Theorem~\ref{rank2} we get for $n\in\ZZ_{\ge0}$
\begin{align}\label{eq:n+3}
X_{n+3}&=X_{I_n}=\sum_{{\bf e}} q^{-\half d_{{\bf e}}^{I_n}}|Gr_{{\bf e}}(I_n)|
X^{(n-1-2e_2,2e_1-n)}\\
&\nonumber=X^{(-n-1,n+2)} + \sum_{p + r \leq
n} {n-r \brack p}_q{n+1-p \brack r}_q X^{(2p-n-1,2r-n)}.
\end{align}

It is known for general $q$ that the coefficients of $X_{-n}$ and $X_{n+3}$ in $\Acal_q(2,2)$ written in terms of the initial cluster $\{X_1,X_2\}$ are given by polynomials in $q$.  Now equations \eqref{eq:n} and \eqref{eq:n+3} are valid for infinitely many values of $q$, thus Proposition \ref{th:XZ-formula} holds for any $q$.

\section{Proof of Theorem~\ref{th:mutations}}\label{pf:mutations}

Let $\Fbb$ be the finite field with $q$ elements.  Suppose that vertex $i$ is a source in the valued quiver $Q=(B,D)$.  Note that the valued quiver $\mu_iQ$
is obtained from $Q$ by reversing all arrows at vertex $i$ so that $i$ is a sink
in $\mu_iQ$ and that the valuations of $\mu_iQ$ equal those of $Q$.  Let
$\Qcal$ denote the Grothendieck group of
$rep~Q$.  Recall that we identify $\Qcal$
with the root system associated to the Cartan counterpart of $B$ with
simple roots $\{\alpha_i\}$ where $\alpha_i=[S_i]\in\Qcal$.  We will abuse notation and also denote by $\Qcal$ the Grothendieck group of $rep~\mu_iQ$ with $\alpha_i=[S'_i]$.
We also denote by $\sigma_i$ the simple reflection associated to $\alpha_i$ in the
Weyl group of $\Qcal$.

Let $rep~Q\langle i\rangle$ denote the full subcategory of
$rep~Q$ of all representations of $Q$ which do not contain
$S_i$ as a direct summand.  In \cite{dlab} it is shown that the reflection
functors $$\Sbb_i^-: rep~Q\leftrightarrow
rep~\mu_iQ:\Sbb_i^+$$ restrict to inverse equivalences
of categories $$\Sbb_i^-: rep~Q\langle
i\rangle\leftrightarrow rep~\mu_iQ\langle
i\rangle:\Sbb_i^+.$$  Since it will be clear from context which to use we will drop the ${}^\pm$ and simply denote both functors by $\Sbb_i$. See \cite{dlab} for precise definitions of these
functors.  We will use the following result proved in \cite{dlab}.

\begin{lemma}\cite[Proposition 2.1]{dlab}\label{sigma-prop}\mbox{}
For $X\in rep~Q\langle i\rangle$ we have $[\Sbb_i X]=\sigma_i([X])$ and $\Sbb_i^2X=X$.
\end{lemma}

We now prove a recursion for the Grassmannians of $rep~Q$, denoted $Gr^Q$, in terms of the Grassmannians of $rep~Q\langle i\rangle$, denoted $Gr^{Q\langle i\rangle}$.  We will use the following convention for Grassmannians of $V\in rep~Q\langle i\rangle$: 
$$Gr_{\bf e}^{Q\langle i\rangle}(V)=\{0\to W\subset V\to V/W\to 0 : [W]={\bf e}; W,V/W\in rep~Q\langle i\rangle\}.$$

\begin{theorem}\label{th:recursion}
Let $M\in rep~\mu_iQ\langle i\rangle$ with $[M]={\bf m}$ and ${\bf e}\in\Qcal$ with ${\bf e}\le {\bf m}$.  Then we have
\[Gr_{\bf e}^{\mu_iQ}(M)=\coprod_{c\ge0} \Fbb^{d_ic(\sigma_i({
\bf e})_i+c)}\times Gr_{c\alpha_i}^Q((m_i-\sigma_i({\bf e})_i-e_i)S_i)\times Gr_{\sigma_i({\bf e}) +c\alpha_i}^Q(\Sbb_iM)\]
where $\sigma_i({\bf e})_j=(1-\delta_{ij})e_j+\delta_{ij}(\sum\limits_{\ell=1}^ne_\ell[b_{i\ell}]_+-e_i)$.
\end{theorem}
\begin{proof}
The main content of the proof is contained in the following lemmas.

\begin{lemma}\label{sink}
For $M\in rep~\mu_iQ\langle i\rangle$ with $[M]={\bf m}$ and ${\bf e}\in \Qcal$ with ${\bf e}\le{\bf m}$ we have $$Gr_{\bf e}^{\mu_iQ}(M)=\coprod_{a\ge0} 
Gr_{a\alpha_i}^{\mu_iQ}((m_i-e_i+a)S'_i)\times Gr_{{\bf e}-a\alpha_i}^{\mu_iQ\langle i\rangle}(M)$$.
\end{lemma}
\begin{proof}
 Consider the map $\zeta:Gr_{\bf e}^{\mu_iQ}(M)\to \coprod\limits_{a\ge0} Gr_{{\bf e}-a\alpha_i}^{\mu_iQ\langle i\rangle}(M)$ given by
$N\oplus aS'_i\mapsto N$.  This map is clearly surjective. Suppose $f:N\hookrightarrow M$ is an element of 
$Gr_{{\bf e}-a\alpha_i}^{\mu_iQ\langle i\rangle}(M)$.  The fibers of $\zeta$ are given by $\zeta^{-1}(N)=\{(f,g):N\oplus aS'_i\hookrightarrow M\}
=\{g:aS'_i\hookrightarrow M/N\}=Gr_{a\alpha_i}^{\mu_iQ}((m_i-e_i+a)S'_i)$.  The result follows.
\end{proof}

\begin{lemma}
 For $M\in rep~Q\langle i\rangle$ with $[M]={\bf m}$ and ${\bf e}\in \Qcal$ with ${\bf e}\le{\bf m}$ we have $$Gr_{\bf e}^Q(M)=\coprod_{d\ge0} 
Gr_{d\alpha_i}^Q((e_i+d)S_i)\times Gr_{{\bf e}+d\alpha_i}^{Q\langle i\rangle}(M)$$.
\end{lemma}
\begin{proof}
 Let $Q^*$ denote the quiver obtained from $Q$ by reversing all the arrows.  Note that vertex $i$ is a sink in $Q^*$.  We will use the same notation for the linear duality functor $?^*=Hom(?,\Fbb):rep~Q\to rep~Q^*$.  The following equalities are immediate: 
\begin{align*} 
Gr_{\bf e}^Q(M)
&=Gr_{{\bf m}-{\bf e}}^{Q^*}(M^*)=\coprod\limits_{d\ge0} 
Gr_{d\alpha_i}^{Q^*}((e_i+d)S_i^*)\times Gr_{{\bf m}-{\bf e}-d\alpha_i}^{Q^*\langle i\rangle}(M^*)\\
&=\coprod\limits_{d\ge0} 
Gr_{e_i\alpha_i}^Q((e_i+d)S_i)\times Gr_{{\bf e}+d\alpha_i}^{Q\langle i\rangle}(M)\\
&=\coprod\limits_{d\ge0} 
Gr_{d\alpha_i}^Q((e_i+d)S_i)\times Gr_{{\bf e}+d\alpha_i}^{Q\langle i\rangle}(M).
\end{align*}
where the second equality follows from Lemma~\ref{sink}.
\end{proof}

The following result is well-known.

\begin{lemma}
 Let $\Fbb$ be a field, $V,W\in Vect_\Fbb$, and $\ell\in\ZZ_{>0}$.  Then $$Gr_\ell(V\oplus W)=\coprod\limits_{a+b=\ell} \Fbb^{a(w-b)}\times Gr_a(V)\times Gr_b(W).$$
\end{lemma}

Putting the preceding three lemmas together we get our recursion.  Let $a=c+d$ and consider the following:
\begin{align*}
&\coprod_{c\ge0} \Fbb^{d_ic(\sigma_i({
\bf e})_i+c)}\times Gr_{c\alpha_i}^Q((m_i-\sigma_i({\bf e})_i-e_i)S_i)\times Gr_{\sigma_i({\bf e}) +c\alpha_i}^Q(\Sbb_iM)\\
&=\coprod\limits_{c\ge0}\coprod\limits_{d\ge0} \Fbb^{d_ic(\sigma_i({
\bf e})_i+c)}\times Gr_{c\alpha_i}^Q((m_i-\sigma_i({\bf e})_i-e_i)S_i)\times Gr_{d\alpha_i}^Q((\sigma_i({\bf e})_i+(c+d)S_i)\\
&\quad\quad\times Gr_{\sigma_i({\bf e}) +(c+d)\alpha_i}^{Q\langle i\rangle}(\Sbb_iM)\\
&=\coprod\limits_{a\ge0}\coprod\limits_{c\ge0} \Fbb^{d_ic(\sigma_i({
\bf e})_i+c)}\times Gr_{c\alpha_i}^Q((m_i-\sigma_i({\bf e})_i-e_i)S_i)\times Gr_{(a-c)\alpha_i}^Q((\sigma_i({\bf e})_i+aS_i)\\
&\quad\quad\times Gr_{\sigma_i({\bf e}) +a\alpha_i}^{Q\langle i\rangle}(\Sbb_iM)\\
&=\coprod\limits_{a\ge0} Gr_{a\alpha_i}^Q((m_i-e_i+a)S_i)\times Gr_{\sigma_i({\bf e}) +a\alpha_i}^{Q\langle i\rangle}(\Sbb_iM)\\
&=\coprod\limits_{a\ge0} Gr_{a\alpha_i}^{\mu_iQ}((m_i-e_i+a)S'_i)\times Gr_{{\bf e}-a\alpha_i}^{\mu_iQ\langle i\rangle}(M)\\
&= Gr_{\bf e}^{\mu_iQ}(M).
\end{align*}
To see the second to last equality note that each of $Gr_{a\alpha_i}^Q((m_i-e_i+a)S_i)$ and 
$Gr_{a\alpha_i}^{\mu_iQ}((m_i-e_i+a)S'_i)$ is just the classical Grassmannian of vector subspaces $Gr_{a}(\Fbb_{d_i}^{m_i-e_i+a})$.
Also note that $Gr_{\sigma_i({\bf e}) +a\alpha_i}^{Q\langle i\rangle}(\Sbb_iM)=Gr_{{\bf e}-a\alpha_i}^{\mu_iQ\langle i\rangle}(M)$
under the equivalence $\Sbb_i$.
\end{proof}

We now show that the recursion on the Grassmannians just obtained matches the recursion in the quantum cluster algebra obtained by mutating the initial cluster, this will prove Theorem~\ref{tool}.
To simplify notation we will use $A_{\bf e}^{Q}(M)=q^{-\half\langle{\bf e},{\bf m}-{\bf e}\rangle}|Gr_{\bf e}^{Q}(M)|$.  We compute the normalized size of the sets in Theorem~\ref{th:recursion} to get: 
$$A_{\bf e}^{\mu_iQ}(M)=\sum_{c\ge0} {m_i-\sigma_i({\bf e})_i-e_i\brack c}_{q^{d_i/2}} A_{\sigma_i({\bf e})+c\alpha_i}^Q(\Sbb_iM).$$

Suppose $N\in rep~Q\langle i\rangle$.   First we expand $X_N$ via the formula \eqref{qcc-formula} to get an element of $\Acal_q(Q)$.
Then we mutate the initial cluster in direction $i$ to get an element of the quantum cluster algebra $\Acal_q(\mu_iQ)$ which turns out to be $X_{\Sbb_iN}$.  This result holds regardless of whether or not $X_N$ is a cluster variable.

This will immediately imply that the same property holds when $i$ is a sink in $Q$ and $M\in rep~Q\langle i\rangle$.  Indeed, assume the result holds when $i$ is a source.  If we begin with $X_{\Sbb_iM}$ in $\Acal_q(\mu_iQ)$ and mutate the initial cluster in direction $i$ we will get $X_{\Sbb_i\Sbb_iM}=X_M$ in the quantum cluster algebra $\Acal_q(Q)$.  But the mutation of clusters is involutive so starting with $X_M$ and mutating the initial cluster in direction $i$ gives $X_{\Sbb_iM}$.

We expand $X_N$ in terms of the initial seed $(\{X_1,\ldots,X'_i,\ldots,X_n\},Q)$:

\begin{align*}
X_N&=\sum_{\bf e} q^{-\half \langle {\bf e}, {\bf n} - {\bf e} \rangle}|Gr_{\bf
e}^Q(N)|X^{({}^*{\bf e}-{\bf e}^*-{}^*{\bf n})}=\sum_{\bf e} A_{\bf
e}^Q(N) X^{({}^*{\bf e}-{\bf e}^*-{}^*{\bf n})}.
\end{align*}

We apply equation \eqref{eq:cl_exp} with ${\bf c}={}^*{\bf e}-{\bf e}^*-{}^*{\bf n}$ to get $X_N$ in terms of the seed
$(\{X_1,\ldots,X_i,\ldots,X_n\},\mu_iQ)$.

\begin{align*}
&\sum_{\bf e} A_{\bf e}^Q(N) X^{({}^*{\bf e}-{\bf e}^*-{}^*{\bf n})}=\sum_{\bf e} \sum\limits_{r\ge0} {c_i \brack r}_{q^{d_i/2}} A_{\bf e}^Q(N) X^{\left(\sum\limits_{\ell=1}^n((1-\delta_{i\ell})c_\ell+(c_i-r)[-b_{\ell i}]_+-c_i\delta_{i\ell})\alpha_\ell\right)}
\end{align*}

Now we substitute ${\bf f}=\sigma_i({\bf
e})+r\alpha_i=\sum\limits_{\ell=1}^n(1-\delta_{i\ell})e_\ell\alpha_\ell+\delta_{
i\ell}\left(\sum_{m=1}^ne_m[b_{im}]_+-e_i+r\right)\alpha_i$ and $\sigma_i({\bf
n})=\sum\limits_{\ell=1}^n(1-\delta_{i\ell})n_\ell\alpha_\ell+\delta_{i\ell}
\left(\sum_{m=1}^nn_m[b_{im}]_+-n_i\right)\alpha_i$ then simplify to get:

\begin{align*}
&\sum_{\bf f} \sum_{r\ge0} {\sigma_i({\bf n})_i-\sigma_i({\bf
f})_i-e_i\brack r}_{q^{d_i/2}}A_{\sigma_i({\bf
f})+r\alpha_i}^Q(\Sbb_i\Sbb_iN)\times\\
&\quad \times X^{\left(\sum\limits_{\ell=1}^n(\sum_{m=1}^n(f_m[-b_{\ell m}']_++(\sigma_i({\bf
n})_m-f_m)[b_{\ell m}']_+)-\sigma_i({\bf n})_\ell)\alpha_\ell\right)}\\
&\nonumber=\sum_{\bf f} A_{\bf
f}^{\Sbb_iN}X^{({}^*{\bf f}-{\bf f}^*-{}^*\sigma_i({\bf n}))}=X_{\Sbb_iN}.
\end{align*}

This completes the proof of Theorem~\ref{tool}.
We now are ready to prove Theorem~\ref{th:mutations}.

\begin{lemma}\label{simple}
Let $Q$ be a valued quiver.  Inside the quantum cluster algebra $\Acal_{|\Fbb|}(Q)$, we have
$X_{[k;k]}^Q=X_{S_k}$ where $S_k$ is the simple representation associated to
vertex $k$ in $Q$.
\end{lemma}
\begin{proof}
First note that $S_k$ has only two subrepresentations $0$ and $S_k$.  So we have
$$X_{S_k}=X^{({}^*{\bf 0}-{\bf
0}^*-{}^*\alpha_k)}+X^{({}^*\alpha_k-\alpha_k^*-{}^*\alpha_k)}=X^{
(-\alpha_k+\sum\limits_{\ell=1}^n[b_{\ell
k}]_+\alpha_\ell)}+X^{(-\alpha_k+\sum\limits_{\ell=1}^n[-b_{\ell
k}]_+\alpha_\ell)}.$$
But the last expression is just the exchange relation defining $X_{[k;k]}$.
\end{proof}

Suppose the seed $({\bf X},Q)$, can be transformed into the
seed $({\bf X}',Q')$, by a sequence of mutations
in directions $k_1$, $k_2$, \ldots, $k_{r+1}$ such that the corresponding sequence of vertices is admissible in $Q$, i.e. $Q'=\mu_{k_{r+1}}\mu_{k_r}\cdots\mu_{k_1}Q$.

We start with the cluster variable $X_{k_{r+1}}'$ in $\Acal_q(Q')$.  This is the cluster variable 
$X_{[k_{r+1};k_{r+1}]}^{\mu_{k_{r+1}}Q'}$ in $\Acal_q(\mu_{k_{r+1}}Q')$.  By Lemma~\ref{simple} we can 
write $X_{[k_{r+1};k_{r+1}]}=X_{S_{k_{r+1}}}$ for $S_{k_{r+1}}\in rep~\mu_{k_{r+1}}Q'$.
Now assume inside the quantum cluster algebra 
$\Acal_q(\mu_{k_{i+1}}\cdots\mu_{k_{r+1}}Q')$ that we have 
$$X_{[k_{r+1};k_{i+1},\ldots, k_{r+1}]}=X_{\Sbb_{k_{i+1}}\cdots\Sbb_{k_r}(S_{k_{r+1}})}$$
for some $i\in[1,r]$, where 
$\Sbb_{k_{i+1}}\cdots\Sbb_{k_r}(S_{k_{r+1}})\in rep~\mu_{k_{i+1}}\cdots\mu_{k_{r+1}}Q'$.  
Notice that this representation is indecomposable and, since the sequence of vertices was admissible, it does not contain 
$S_{k_i}$ as a direct summand.  Thus mutating the initial cluster in direction $k_i$ gives 
$$X_{[k_{r+1};k_i,k_{i+1},\ldots, k_{r+1}]}=X_{\Sbb_{k_i}\Sbb_{k_{i+1}}\cdots\Sbb_{k_r}(S_{k_{r+1}})}$$ in the 
quantum cluster algebra $\Acal_q(\mu_{k_i}\mu_{k_{i+1}}\cdots\mu_{k_{r+1}}Q')$.
So by induction and the fact that mutations are involutive we have inside the quantum cluster algebra 
$\Acal_q(\mu_{k_1}\cdots\mu_{k_{r+1}}Q')=\Acal_q(Q)$ the equality 
$$X_{[k_{r+1};k_1,k_2,\ldots,k_{r+1}]}=X_{\Sbb_{k_1}\cdots\Sbb_{k_r}(S_{k_{r+1}})}$$ where 
$\Sbb_{k_1}\cdots\Sbb_{k_r}(S_{k_{r+1}})\in rep~\mu_{k_1}\cdots\mu_{k_{r+1}}Q'=rep~Q$.
\noindent This completes the proof.

\end{document}